\documentclass[12pt,oneside,english]{amsart}
\usepackage[T1]{fontenc}
\usepackage[latin9]{inputenc}
\usepackage{geometry}
\usepackage{amsthm}
\usepackage{amsbsy}
\usepackage{amstext}
\usepackage{amssymb}

\makeatletter
\numberwithin{equation}{section}
\numberwithin{figure}{section}
\theoremstyle{plain}
\newtheorem{thm}{\protect\theoremname}
  \theoremstyle{plain}
  \newtheorem{lem}[thm]{\protect\lemmaname}

\makeatother

\usepackage{babel}
  \providecommand{\lemmaname}{Lemma}
\providecommand{\theoremname}{Theorem}

\begin{document}

\title{Integers with a predetermined prime factorization}
\begin{abstract}
A classic question in analytic number theory is to find asymptotics
for $\sigma_{k}(x)$ and $\pi_{k}(x)$, the number of integers $n\leq x$
with exactly $k$ prime factors, where $\pi_{k}(x)$ has the added
constraint that all the factors are distinct. This problem was originally
resolved by Landau in 1900, and much work was subsequently done where
$k$ is allowed to vary. In this paper we look at a similar question
about integers with a specific prime factorization. Given $\boldsymbol{\alpha}\in\mathbb{N}^{k}$,
$\boldsymbol{\alpha}=\left(\alpha_{1},\alpha_{2},\dots,\alpha_{k}\right)$
let $\sigma_{\boldsymbol{\alpha}}(x)$ denote the number of integers
of the form $n=p_{1}^{\alpha_{1}}\cdots p_{k}^{\alpha_{k}}$ where
the $p_{i}$ are not necessarily distinct, and let $\pi_{\boldsymbol{\alpha}}(x)$ denote the same counting function with the added condition that the factors are distinct. Our main result
is asymptotics for both of these functions.
\end{abstract}

\author{Eric Naslund}

\maketitle

\section{Introduction}

One of the major problems in the 19th century was to find the growth
rate of the number of primes less then $x$, that is the function
\[
\pi(x):=\sum_{p\leq x}1.
\]
In 1797, Legendre conjectured that $\pi(x)$ is asymptotic to $\frac{x}{\log x}$,
written as $\pi(x)\sim\frac{x}{\log x}$, which means that we have
the limit 
\[
\lim_{x\rightarrow\infty}\frac{\pi(x)}{x/\log x}=1.
\]
Although a more precise conjecture was given by Gauss, little progress
was made over the next 50 years. In 1848 and 1850, Chebyshev made
several contributions, and managed to prove weaker upper and lower
bounds. A major breakthrough occurred in 1859, when Riemann published
his seminal paper, {}``On the Number of Primes Less Than a Given
Magnitude,'' in which he outlined a proof of Legendre's conjecture
using complex analysis and the zeta function. In 1896, 99 years after
Legendre made his conjecture, Hadamard and de la Vall\'{e}e Poussin rigorously
completed Riemann's outline, proving what is known today as the prime
number theorem \cite{BookMontVaughn}. In particular, we can write
down the explicit error term : 
\begin{equation}
\pi(x)=\frac{x}{\log x}+O\left(\frac{x}{\log^{2}x}\right),\label{eq:Prime number theorem}
\end{equation}
 but to be more precise than this we would need to introduce the function
from Gauss's conjecture.

A natural follow up question is whether or not we have similar asymptotics
for the number of integers with exactly $k$ prime factors. There
are two reasonable ways to define the counting function; let $\sigma_{k}(x)$
denote the number of integers less then $x$ with exactly $k$ prime
factors, and let $\pi_{k}(x)$ be the same but with the added constraint
that the $k$ prime factors must be distinct. For convenience, we
also define the sets $\mathcal{P}_{k}^{\sigma}=\left\{ n:\ n=p_{1}\cdots p_{k}\right\} $
and $\mathcal{P}_{k}^{\pi}=\left\{ n:\ n=p_{1}\cdots p_{k}\ \text{where}\ i\neq j\Rightarrow p_{i}\neq p_{j}\right\} $,
so that we may write 
\[
\sigma_{k}(x)=\sum_{\begin{array}{c}
n\leq x\\
n\in\mathcal{P}_{k}^{\sigma}
\end{array}}1\ \ \ \text{and}\ \ \ \pi_{k}(x)=\sum_{\begin{array}{c}
n\leq x\\
n\in\mathcal{P}_{k}^{\pi}
\end{array}}1.
\]
In $1900$ by Landau \cite{Landau1900} found the growth rate of these functions, and he proved that for fixed $k$ we have
\begin{equation}
\pi_{k}(x)\sim\sigma_{k}(x)\sim\frac{x\left(\log\log x\right)^{k-1}}{(k-1)!\log x}.\label{eq:pi/sigma Landau asymptotic}
\end{equation}
E. M. Wright then gave a short elementary proof of this in 1954 \cite{Wright1954}.
Heuristically we might expect this kind of asymptotic since $\sum_{k=1}^{\infty}\sigma_{k}(x)=\lfloor x\rfloor$,
and if we could ignore the error term and sum over all $k\leq\log x$,
we would arrive back at this equality again as 
\[
\sum_{k=1}^{\infty}\sigma_{k}(x)\approx\sum_{k=1}^{\infty}\frac{x\left(\log\log x\right)^{k-1}}{(k-1)!\log x}=\frac{x}{\log x}\sum_{k=0}^{\infty}\frac{\left(\log\log x\right)^{k}}{k!}=x.
\]
Note that even though this works out, the heuristic is not entirely
reliable. It seems to suggest that $\sigma_{k}(x)\sim\frac{x\left(\log\log x\right)^{k-1}}{(k-1)!\log x}$
even when $k$ varies with $x$, which is not true when $k\approx\log\log x$
\cite{HildTenenAldof}. In his paper, Landau also gave explicit
error terms, and showed that for $k\geq2$ 
\begin{equation}
\sigma_{k}(x)=\frac{x\left(\log\log x\right)^{k-1}}{(k-1)!\log x}+O\left(\frac{x\left(\log\log x\right)^{k-2}}{\log x}\right)\label{eq:Landau sigma}
\end{equation}
and 
\begin{equation}
\pi_{k}(x)=\frac{x\left(\log\log x\right)^{k-1}}{(k-1)!\log x}+O\left(\frac{x\left(\log\log x\right)^{k-2}}{\log x}\right)\label{eq:landau pi}
\end{equation}
where the notation $O(f(x))$ means that the error term is bounded
in absolute value by some constant multiple of $f(x)$.  (Although seperated on different lines, note that the above asymptotics are  indeed the same.)  In this paper we are
interested in something very similar, which is counting the number
of integers of a particular shape, integers of the form $p_{1}^{\alpha_{1}}\cdots p_{n}^{\alpha_{n}}$
where the $\alpha_{i}$ are fixed exponents. For example, we may ask
how many integers of the form $pq^{3}$ are there less than $x$.
To discuss this problem, we begin by introducing some notation. Given
a vector $\boldsymbol{\alpha}=\left(\alpha_{1},\cdots,\alpha_{k}\right)\in\mathbb{N}^{k}$,
define $\sigma_{\boldsymbol{\alpha}}(x)$ to be the number of integers
$n\leq x$ of the form $n=p_{1}^{\alpha_{1}}\cdots p_{k}^{\alpha_{k}}$,
allowing prime repetitions, and $\pi_{\boldsymbol{\alpha}}(x)$ to
be the number without prime repetitions. If we set $\mathcal{P}_{\boldsymbol{\alpha}}^{\sigma}=\left\{ n:\ n=p_{1}^{\alpha_{1}}\cdots p_{r}^{\alpha_{r}}\right\} $,
and $\mathcal{P}_{\boldsymbol{\alpha}}^{\pi}=\left\{ n:\ n=p_{1}^{\alpha_{1}}\cdots p_{r}^{\alpha_{r}}\ \text{where}\ i\neq j\Rightarrow p_{i}\neq p_{j}\right\} $,
then as was done for $\pi_{k}(x)$, and $\sigma_{k}(x)$, we can rewrite
these counting functions as 
\[
\sigma_{\boldsymbol{\alpha}}(x)=\sum_{\begin{array}{c}
n\leq x\\
n\in\mathcal{P}_{\boldsymbol{\alpha}}^{\sigma}
\end{array}}1\ \ \ \text{and}\ \ \ \pi_{\boldsymbol{\alpha}}(x)=\sum_{\begin{array}{c}
n\leq x\\
n\in\mathcal{P}_{\boldsymbol{\alpha}}^{\pi}
\end{array}}1.
\]
Our goal is to provide asymptotics for $\sigma_{\boldsymbol{\alpha}}(x)$ and $\pi_{\boldsymbol{\alpha}}(x)$,
and our main theorem is:
\begin{thm}
\label{thm: Main Theorem}Let $r,\alpha$ be positive integers. Suppose
we have a vector of the form $\boldsymbol{\alpha}=\left(\alpha,\cdots,\alpha,\alpha_{1},\cdots,\alpha_{r}\right)\in\mathbb{N}^{k+r}$,
where $k>0$ is the multiplicity of $\alpha$, and where $\alpha<\alpha_{i}$
for all $i$. Then if $\boldsymbol{\beta}=\left(\alpha_{1},\cdots,\alpha_{r}\right)\in\mathbb{N}^{r}$,
we have 
\[
\sigma_{\boldsymbol{\alpha}}\left(x\right)\sim\sigma_{k}\left(x^{\frac{1}{\alpha}}\right)\sum_{n\in\mathcal{P}_{\boldsymbol{\beta}}^{\sigma}}n^{-\frac{1}{\alpha}}
\]
and
\[
\pi_{\boldsymbol{\alpha}}\left(x\right)\sim\sigma_{k}\left(x^{\frac{1}{\alpha}}\right)\sum_{n\in\mathcal{P}_{\boldsymbol{\beta}}^{\pi}}n^{-\frac{1}{\alpha}}.
\]

\end{thm}
\noindent The above theorem tells us that the higher powers introduce a constant factor into the asymptotic since both of the series $\sum_{n\in\mathcal{P}_{\boldsymbol{\beta}}^{\sigma}}n^{-\frac{1}{\alpha}}$
and $\sum_{n\in\mathcal{P}_{\boldsymbol{\beta}}^{\pi}}n^{-\frac{1}{\alpha}}$
converge absolutely.   The convergence of these series follows from the fact that $\frac{\alpha_{i}}{\alpha}>1$
along with equation \ref{eq:sigma beta upper bound} in the next section.
In particular, returning to our previous example of counting the number
of integers of the form $pq^{3}$ less than $x$, we have that $\mathcal{P}_{\boldsymbol{\beta}}^{\pi}=\mathcal{P}_{\boldsymbol{\beta}}^{\sigma}=\left\{ p^{3}:\ p \text{ is prime}\right\} $,
and hence 
\[
\pi_{\left(1,3\right)}(x)\sim\sigma_{\left(1,3\right)}(x)\sim\frac{x}{\log x}\sum_{p}\frac{1}{p^{3}}=\frac{x}{\log x}P(3)
\]
where $P(s)=\sum_{p}p^{-s}$ is the prime zeta function. We can ask
whether the constant can always be rewritten as a product of prime
zeta functions, and this is answered by the following theorem:
\begin{thm}
\label{thm:Theorem 2}Suppose we are given $\alpha<\alpha_{1}\leq\cdots\leq\alpha_{r}$,
and that for any choice of $\epsilon_{i}\in\left\{ -1,0,1\right\} $,
we have $\sum_{i}\epsilon_{i}\alpha_{i}=0$ implies $\epsilon_{i}=0$
for every $i$. Then 
\[
\sum_{n\in\mathcal{P}_{\boldsymbol{\beta}}^{\sigma}}n^{-\frac{1}{\alpha}}=\prod_{i=1}^{r}P\left(\frac{\alpha_{i}}{\alpha}\right)
\]
where $P(s)=\sum_{p}p^{-s}$ is the prime zeta function. This is equivalent
to the condition that every $n\in\mathcal{P}_{\boldsymbol{\beta}}^{\sigma}$,
where $\boldsymbol{\beta}=\left(\alpha_{1},\dots,\alpha_{r}\right)$,
has a unique representation as $n=p_{1}^{\alpha_{1}}\cdots p_{r}^{\alpha_{r}}$.
\end{thm}
\noindent For example, the above two theorems imply that the number of integers
of the form $n=p_{1}p_{2}p_{3}^{3}p_{4}^{5}p_{5}^{19}$, with $n\leq x$,
will be asymptotic to 
\[
\sigma_{2}\left(x\right)P(3)P(5)P(19)\sim\frac{x\log\log x}{\log x}P(3)P(5)P(19).
\]

\section{The Main Result}

It is very important to split up the smallest power, as this is contributes
the most to the sum. Throughout this section, we write our vector
of exponents as $\boldsymbol{\alpha}=\left(\alpha,\cdots,\alpha,\alpha_{1},\cdots,\alpha_{r}\right)\in\mathbb{N}^{k+r}$,
with $1\leq\alpha<\alpha_{1}\leq\cdots\leq\alpha_{r}$, where $k>0$
is the multiplicity of $\alpha$, and let $\boldsymbol{\beta}=\left(\alpha_{1},\cdots,\alpha_{r}\right)\in\mathbb{N}^{r}$.
To start, we provide a simple upper bound for $\sigma_{\boldsymbol{\beta}}(x)$.
Notice that 
\[
\pi_{\boldsymbol{\beta}}(x)\leq\sigma_{\boldsymbol{\beta}}(x)=\sum_{\begin{array}{c}
m\leq x\\
m\in\mathcal{P}_{\boldsymbol{\beta}}^{\sigma}
\end{array}}1\leq\sum_{p_{1}^{\alpha_{1}}\cdots p_{r}^{\alpha_{r}}\leq x}1,
\]
where the right hand sum ranges over all vectors of primes of length
$r$ satisfying $p_{1}^{\alpha_{1}}\cdots p_{r}^{\alpha_{r}}\leq x$.
Since $\alpha_{1}\leq\alpha_{i}$ for all $i$, and $p_{1}^{\alpha_{1}}\cdots p_{r}^{\alpha_{r}}\leq x$ implies that $p_{1}^{\alpha_{1}}p_{2}^{\alpha_{1}}\cdots p_{r}^{\alpha_{1}}\leq x$, we see that replacing every exponent by $\alpha_{1}$ only increases
the sum. Then using \ref{eq:pi/sigma Landau asymptotic} we have 
\begin{equation}
\pi_{\boldsymbol{\beta}}(x)\leq\sigma_{\boldsymbol{\beta}}(x)\leq\sum_{p_{1}\cdots p_{r}\leq x^{\frac{1}{\alpha_{1}}}}1=O\left(x^{\frac{1}{\alpha_{1}}}\frac{\left(\log\log x\right)^{r-1}}{\log x}\right).\label{eq:sigma beta upper bound}
\end{equation}
The following subsection is devoted to examining $\sigma_{\boldsymbol{\alpha}}(x)$.
The key will be using the hyperbola method, and most of the lemmas
will apply identically to the proof for $\pi_{\boldsymbol{\alpha}}(x)$.

\subsection{$\sigma_{\boldsymbol{\alpha}}(x)$}

Each integer $n\in\mathcal{P}_{\boldsymbol{\alpha}}^{\sigma}$ has
one part in $\mathcal{P}_{k}^{\sigma}$, and one part in $\mathcal{P}_{\boldsymbol{\beta}}^{\sigma}$,
and our goal will be to split it up between these two to better understand
$\sigma_{\boldsymbol{\alpha}}(x)$. With this in mind, we might expect
\[
\sigma_{\boldsymbol{\alpha}}(x)\approx\sum_{\begin{array}{c}
mn^{\alpha}\leq x\\
n\in\mathcal{P}_{k}^{\sigma},m\in\mathcal{P}_{\boldsymbol{\beta}}^{\sigma}
\end{array}}1.
\]
However, this will not be an exact equality as an integer $k\leq x$
with $k\in\mathcal{P}_{\boldsymbol{\alpha}}^{\sigma}$ may have more
than one representation of the form $k=mn^{\alpha}$ with $n\in\mathcal{P}_{k}^{\sigma},\ m\in\mathcal{P}_{\boldsymbol{\beta}}^{\sigma}$.
Since $k\in\mathcal{P}_{\boldsymbol{\alpha}}^{\sigma}$ can have at
most one representation of the form $k=mn^{\alpha}$ with $n\in\mathcal{P}_{k}^{\pi},\ m\in\mathcal{P}_{\boldsymbol{\beta}}^{\sigma}$,
we have the inequalities
\[
\sum_{\begin{array}{c}
mn^{\alpha}\leq x\\
n\in\mathcal{P}_{k}^{\pi},m\in\mathcal{P}_{\boldsymbol{\beta}}^{\sigma}
\end{array}}1\leq\sigma_{\boldsymbol{\alpha}}(x)\leq\sum_{\begin{array}{c}
mn^{\alpha}\leq x\\
n\in\mathcal{P}_{k}^{\sigma},m\in\mathcal{P}_{\boldsymbol{\beta}}^{\sigma}
\end{array}}1.
\]
Rewriting so that we first sum over $m$, this is 
\[
\sum_{\begin{array}{c}
mn^{\alpha}\leq x\\
n\in\mathcal{P}_{k}^{\sigma},m\in\mathcal{P}_{\boldsymbol{\beta}}^{\sigma}
\end{array}}1=\sum_{\begin{array}{c}
m\leq x\\
m\in\mathcal{P}_{\boldsymbol{\beta}}^{\sigma}
\end{array}}\sum_{\begin{array}{c}
n^{\alpha}\leq\frac{x}{m}\\
n\in\mathcal{P}_{k}^{\sigma}
\end{array}}1=\sum_{\begin{array}{c}
m\leq x\\
m\in\mathcal{P}_{\boldsymbol{\beta}}^{\sigma}
\end{array}}\sigma_{k}\left(\left(\frac{x}{m}\right)^{\frac{1}{\alpha}}\right)
\]
and we have that 
\begin{equation}
\sum_{\begin{array}{c}
m\leq x\\
m\in\mathcal{P}_{\boldsymbol{\beta}}^{\sigma}
\end{array}}\pi_{k}\left(\left(\frac{x}{m}\right)^{\frac{1}{\alpha}}\right)\leq\sigma_{\boldsymbol{\alpha}}(x)\leq\sum_{\begin{array}{c}
m\leq x\\
m\in\mathcal{P}_{\boldsymbol{\beta}}^{\sigma}
\end{array}}\sigma_{k}\left(\left(\frac{x}{m}\right)^{\frac{1}{\alpha}}\right).\label{eq: upper+low sigma bounds}
\end{equation}
Our first goal will be to remove all of the terms from the sum with $m\geq\frac{x}{\left(\log x\right)^{C}}$
for some constant $C>2$, without introducing large error. For example, we could take $C=3$ to prove the asymptotic.  However to achieve the optimal error term we need something of the form
$C=2\alpha\alpha_{1}+1$, a choice which will become clear later on.
Note that we need only bound this sum for $\sigma_{k}(x)$, since
$\pi_{k}(x)\leq\sigma_{k}(x)$, and this is covered by the following
lemma.
\begin{lem}
\label{lem: truncating sigma alpha series}For $C>1$ we have that
\[
\sum_{\begin{array}{c}
\left(\log x\right)^{C}<m\leq x\\
m\in\mathcal{P}_{\boldsymbol{\beta}}^{\sigma}
\end{array}}\sigma_{k}\left(\left(\frac{x}{m}\right)^{\frac{1}{\alpha}}\right)=O\left(\frac{x^{\frac{1}{\alpha}}}{\left(\log x\right)^{\left(C-1\right)\left(1-\frac{\alpha}{\alpha_{1}}\right)}}\right).
\]
\end{lem}
\begin{proof}
We may change the order of summation and write 
\begin{eqnarray*}
\sum_{\begin{array}{c}
\left(\log x\right)^{C}\leq m\leq x\\
m\in\mathcal{P}_{\boldsymbol{\beta}}^{\sigma}
\end{array}}\sigma_{k}\left(\left(\frac{x}{m}\right)^{\frac{1}{\alpha}}\right) & = & \sum_{\begin{array}{c}
\left(\log x\right)^{C}\leq m\leq x\\
m\in\mathcal{P}_{\boldsymbol{\beta}}^{\sigma}
\end{array}}\sum_{\begin{array}{c}
n^{\alpha}\leq\frac{x}{m}\\
n\in\mathcal{P}_{k}^{\sigma}
\end{array}}1\\
 & = & \sum_{\begin{array}{c}
n^{\alpha}\leq\frac{x}{\left(\log x\right)^{C}}\\
n\in\mathcal{P}_{k}^{\sigma}
\end{array}}\sum_{\begin{array}{c}
\left(\log x\right)^{C}\leq m\leq\frac{x}{n^{\alpha}}\\
m\in\mathcal{P}_{\boldsymbol{\beta}}^{\sigma}
\end{array}}1.
\end{eqnarray*}
Using \ref{eq:sigma beta upper bound} this is bounded above by 
\begin{eqnarray*}
\sum_{\begin{array}{c}
n^{\alpha}\leq\frac{x}{\left(\log x\right)^{C}}\\
n\in\mathcal{P}_{k}^{\sigma}
\end{array}}\sum_{\begin{array}{c}
m\leq\frac{x}{n^{\alpha}}\\
m\in\mathcal{P}_{\boldsymbol{\beta}}^{\sigma}
\end{array}}1 & = & \sum_{\begin{array}{c}
n^{\alpha}\leq\frac{x}{\left(\log x\right)^{C}}\\
n\in\mathcal{P}_{k}^{\sigma}
\end{array}}O\left(\frac{x^{\frac{1}{\alpha_{1}}}}{n^{\frac{\alpha}{\alpha_{1}}}}\frac{\left(\log\log\left(x/n^{\alpha}\right)\right)^{r-1}}{\log\left(x/n^{\alpha}\right)}\right)\\
 & = & O\left(x^{\frac{1}{\alpha_{1}}}\left(\log\log x\right)^{r-1}\sum_{\begin{array}{c}
n^{\alpha}\leq\frac{x}{\left(\log x\right)^{C}}\\
n\in\mathcal{P}_{k}^{\sigma}
\end{array}}\frac{1}{n^{\frac{\alpha}{\alpha_{1}}}}\right).
\end{eqnarray*}
Taking the trivial bound, the inner sum becomes 
\begin{eqnarray*}
\sum_{\begin{array}{c}
n^{\alpha}\leq\frac{x}{\left(\log x\right)^{C}}\\
n\in\mathcal{P}_{k}^{\sigma}
\end{array}}\frac{1}{n^{\frac{\alpha}{\alpha_{1}}}} & \leq & \sum_{n\leq\frac{x^{\frac{1}{\alpha}}}{\log^{\frac{C}{\alpha}}x}}\frac{1}{n^{\frac{\alpha}{\alpha_{1}}}}=O\left(\left(\frac{x^{\frac{1}{\alpha}}}{\log^{C}x}\right)^{-\frac{\alpha}{\alpha_{1}}-1}\right)\\
 & = & O\left(\frac{1}{\left(\log x\right)^{C\left(1-\frac{\alpha}{\alpha_{1}}\right)}}\right),
\end{eqnarray*}
so that we have the upper bound 
\[
O\left(\frac{x^{\frac{1}{\alpha}}}{\left(\log x\right)^{\left(C-1\right)\left(1-\frac{\alpha}{\alpha_{1}}\right)}}\frac{\left(\log\log x\right)^{r-1}}{\left(\log x\right)^{1-\frac{\alpha}{\alpha_{1}}}}\right)=O\left(\frac{x^{\frac{1}{\alpha}}}{\left(\log x\right)^{\left(C-1\right)\left(1-\frac{\alpha}{\alpha_{1}}\right)}}\right).
\]

\end{proof}
\noindent Combining \ref{eq: upper+low sigma bounds} along with Lemma \ref{lem: truncating sigma alpha series}
and Landau's estimates \ref{eq:Landau sigma}, \ref{eq:landau pi}
for $k>1$ yields
\begin{equation}
\begin{split}\sigma_{\boldsymbol{\alpha}}(x)= & \frac{1}{(k-1)!}\sum_{\begin{array}{c}
m\leq\left(\log x\right)^{C}\\
m\in\mathcal{P}_{\boldsymbol{\beta}}^{\sigma}
\end{array}}\alpha\frac{x^{\frac{1}{\alpha}}\left(\log\left(\frac{1}{\alpha}\log\left(\frac{x}{m}\right)\right)\right)^{k-1}}{m^{\frac{1}{\alpha}}\log\left(\frac{x}{m}\right)}\\
 & +O\left(\frac{x^{\frac{1}{\alpha}}}{\left(\log x\right)^{\left(C-1\right)\left(1-\frac{\alpha}{\alpha_{1}}\right)}}+\sum_{\begin{array}{c}
m\leq\left(\log x\right)^{C}\\
m\in\mathcal{P}_{\boldsymbol{\beta}}^{\sigma}
\end{array}}\frac{x^{\frac{1}{\alpha}}\left(\log\left(\frac{1}{\alpha}\log\left(\frac{x}{m}\right)\right)\right)^{k-2}}{m^{\frac{1}{\alpha}}\log\left(\frac{x}{m}\right)}\right)
\end{split}
,\label{eq:k>1 pre asymptotic}
\end{equation}
and for $k=1$ by  \ref{eq:Prime number theorem}, the prime number theorem,
we have 
\begin{equation}
\begin{split}\sigma_{\boldsymbol{\alpha}}(x)= & \sum_{\begin{array}{c}
m\leq\left(\log x\right)^{C}\\
m\in\mathcal{P}_{\boldsymbol{\beta}}^{\sigma}
\end{array}}\alpha\frac{x^{\frac{1}{\alpha}}}{m^{\frac{1}{\alpha}}\log\left(\frac{x}{m}\right)}\\
 & +O\left(\frac{x^{\frac{1}{\alpha}}}{\left(\log x\right)^{\left(C-1\right)\left(1-\frac{\alpha}{\alpha_{1}}\right)}}+\sum_{\begin{array}{c}
m\leq\left(\log x\right)^{C}\\
m\in\mathcal{P}_{\boldsymbol{\beta}}^{\sigma}
\end{array}}\frac{x^{\frac{1}{\alpha}}}{m^{\frac{1}{\alpha}}\log^{2}\left(\frac{x}{m}\right)}\right)
\end{split}
.\label{eq:k=00003D1 pre asymptotic}
\end{equation}
If we write $\left(\log\left(\frac{1}{\alpha}\log\left(\frac{x}{m}\right)\right)\right)^{k-1}=\left(\log\log\left(\frac{x}{m}\right)-\log\alpha\right)^{k-1}$
and then expand using the binomial theorem, all of the terms will
be consumed by the error term except for the one with $\left(\log\log\left(\frac{x}{m}\right)\right)^{k-1}$,
which allows us to change the main term in the above to 
\begin{equation} \label{main rewrite}
\frac{1}{(k-1)!}\sum_{\begin{array}{c}
m\leq\left(\log x\right)^{C}\\
m\in\mathcal{P}_{\boldsymbol{\beta}}^{\sigma}
\end{array}}\alpha\frac{x^{\frac{1}{\alpha}}\left(\log\log\left(\frac{x}{m}\right)\right)^{k-1}}{m^{\frac{1}{\alpha}}\log\left(\frac{x}{m}\right)}.
\end{equation}
We may clean up the error terms by bounding each part of the sum from
above. Since $m\leq\left(\log x\right)^{C}$, $\frac{1}{\log\left(\frac{x}{m}\right)}$
is bounded above by 
\[
\frac{1}{\log\left(\frac{x}{\left(\log x\right)^{C}}\right)}=\frac{1}{\log\left(x\right)-C\log\log x}=\frac{1}{\log x}+O\left(\frac{\log\log x}{\log^{2}x}\right).
\]
We also have the trivial bounds 
\[
\log\left(\frac{1}{\alpha}\log\left(\frac{x}{m}\right)\right)\leq\left(\log\left(\log\left(x\right)\right)\right),
\]
and 
\[
\sum_{\begin{array}{c}
m\leq\left(\log x\right)^{C}\\
m\in\mathcal{P}_{\boldsymbol{\beta}}^{\sigma}
\end{array}}\frac{1}{m^{\frac{1}{\alpha}}}\leq\sum_{m\in\mathcal{P}_{\boldsymbol{\beta}}^{\sigma}}m^{-\frac{1}{\alpha}}
\]
since the right hand side is a convergent series. Combining these, for integers $A\geq 0$, $B>1$
we have that
\begin{equation}
\sum_{\begin{array}{c}
m\leq\left(\log x\right)^{C}\\
m\in\mathcal{P}_{\boldsymbol{\beta}}^{\sigma}
\end{array}}\frac{x^{\frac{1}{\alpha}}\left(\log\left(\frac{1}{\alpha}\log\left(\frac{x}{m}\right)\right)\right)^{A}}{m^{\frac{1}{\alpha}}\log^{B}\left(\frac{x}{M}\right)}=O\left(\frac{x^{\frac{1}{\alpha}}\left(\log\log x\right)^{A}}{\log^{B}\left(x\right)}\right),\label{eq:Error term upper bound}
\end{equation}
which gives an upper bound on the error term in both cases, $k=1$
and $k>1$. The following lemma allows us to deal with the main term:
\begin{lem}
\label{lem: trunc part asymp}For $C>1$, we have that 
\[
\sum_{\begin{array}{c}
m\leq\left(\log x\right)^{C}\\
m\in\mathcal{P}_{\boldsymbol{\beta}}^{\sigma}
\end{array}}\frac{\left(\log\log\left(\frac{x}{m}\right)\right)^{k-1}}{m^{\frac{1}{\alpha}}\log\left(\frac{x}{m}\right)}=\frac{\left(\log\log\left(x\right)\right)^{k-1}}{\log x}\sum_{\begin{array}{c}
m\leq\left(\log x\right)^{C}\\
m\in\mathcal{P}_{\boldsymbol{\beta}}^{\sigma}
\end{array}}m^{-\frac{1}{\alpha}}+O\left(\frac{\left(\log\log x\right)^{k-1}}{\log^{2}x}\right).
\]
\end{lem}
\begin{proof}
First, note that we have the bounds 
\[
\frac{1}{\log\left(x\right)}\leq\frac{1}{\log\left(\frac{x}{m}\right)}\leq\frac{1}{\log\left(\frac{x}{\log x}\right)}
\]
and 
\[
\left(\log\log\left(\frac{x}{\log x}\right)\right)^{k-1}\leq\left(\log\log\left(\frac{x}{m}\right)\right)^{k-1}\leq\left(\log\log\left(x\right)\right)^{k-1}.
\]
Using power series expansions we may write 
\[
\frac{1}{\log\left(\frac{x}{\log x}\right)}=\frac{1}{\log\left(x\right)\left(1-\frac{\log\log x}{\log x}\right)}=\frac{1}{\log x}+O\left(\frac{\log\log x}{\log^{2}x}\right)
\]
and
\[
\left(\log\log\left(\frac{x}{\log x}\right)\right)^{k-1}=\left(\log\log x+\log\left(1-\frac{\log\log x}{\log x}\right)\right)^{k-1}=\left(\log\log x\right)^{k-1}+O\left(\frac{\left(\log\log x\right)^{k-1}}{\log x}\right).
\]
Then \ref{eq:Error term upper bound} implies that 
\[
\sum_{\begin{array}{c}
m\leq\left(\log x\right)^{C}\\
m\in\mathcal{P}_{\boldsymbol{\beta}}^{\sigma}
\end{array}}\frac{\left(\log\log\left(\frac{x}{m}\right)\right)^{k-1}}{m^{\frac{1}{\alpha}}\log\left(\frac{x}{m}\right)}=\frac{\left(\log\log x\right)^{k-1}}{\log x}\sum_{\begin{array}{c}
m\leq\left(\log x\right)^{C}\\
m\in\mathcal{P}_{\boldsymbol{\beta}}^{\sigma}
\end{array}}m^{-\frac{1}{\alpha}}+O\left(\frac{\left(\log\log x\right)^{k-1}}{\log^{2}x}\right).
\]

\end{proof}
Let $C=2\alpha\alpha_{1}+1$ so that $\left(C-1\right)\left(1-\frac{\alpha}{\alpha_{1}}\right)=2\alpha \left(\alpha_{1}-\alpha\right)\geq2$.
Upon combining \ref{eq:k>1 pre asymptotic}, \ref{main rewrite}, \ref{eq:Error term upper bound},
and lemma \ref{lem: trunc part asymp} for $k>1$ we obtain 
\begin{equation}
\sigma_{\boldsymbol{\alpha}}(x)=\alpha\frac{x^{\frac{1}{\alpha}}\left(\log\log x\right)^{k-1}}{(k-1)!\log x}\sum_{\begin{array}{c}
m\leq\left(\log x\right)^{C}\\
m\in\mathcal{P}_{\boldsymbol{\beta}}^{\sigma}
\end{array}}m^{-\frac{1}{\alpha}}+O\left(x^{\frac{1}{\alpha}}\frac{\left(\log\log x\right)^{k-2}}{\log x}\right).\label{eq:sigma alpha asymptotic ALMOST}
\end{equation}
Similarly, \ref{eq:k=00003D1 pre asymptotic}, \ref{eq:Error term upper bound},
and lemma \ref{lem: trunc part asymp} together yield

\[
\sigma_{\boldsymbol{\alpha}}(x)=\alpha\frac{x^{\frac{1}{\alpha}}}{\log x}\sum_{\begin{array}{c}
m\leq\left(\log x\right)^{C}\\
m\in\mathcal{P}_{\boldsymbol{\beta}}^{\sigma}
\end{array}}m^{-\frac{1}{\alpha}}+O\left(\frac{x^{\frac{1}{\alpha}}}{\log^{2}x}\right)
\]
for $k=1$. To deal with the last sum, write
\[
\sum_{\begin{array}{c}
m\leq\left(\log x\right)^{C}\\
m\in\mathcal{P}_{\boldsymbol{\beta}}^{\sigma}
\end{array}}m^{-\frac{1}{\alpha}}=\sum_{m\in\mathcal{P}_{\boldsymbol{\beta}}^{\sigma}}m^{-\frac{1}{\alpha}}-\sum_{\begin{array}{c}
m>\left(\log x\right)^{C}\\
m\in\mathcal{P}_{\boldsymbol{\beta}}^{\sigma}
\end{array}}m^{-\frac{1}{\alpha}}.
\]
Applying summation by parts, we have that 
\begin{eqnarray*}
\sum_{\begin{array}{c}
m>\left(\log x\right)^{C}\\
m\in\mathcal{P}_{\boldsymbol{\beta}}^{\sigma}
\end{array}}m^{-\frac{1}{\alpha}} & = & \int_{\left(\log x\right)^{C}}^{\infty}t^{-\frac{1}{\alpha}}d\left(\sigma_{\boldsymbol{\beta}}(t)\right)\\
 & = & t^{-\frac{1}{\alpha}}\sigma_{\boldsymbol{\beta}}(t)\biggr|_{\left(\log x\right)^{C}}^{\infty}+\frac{1}{\alpha}\int_{\left(\log x\right)^{C}}^{\infty}t^{-\frac{1}{\alpha}-1}\sigma_{\boldsymbol{\beta}}(t)dt.
\end{eqnarray*}
Then by \ref{eq:sigma beta upper bound} this becomes 
\[
O\left(\left(\log x\right)^{C\left(\frac{1}{\alpha_{1}}-\frac{1}{\alpha}\right)}\left(\log\log\log x\right)^{r-1}\right)=O\left(\frac{1}{\left(\log x\right)^{2}}\right)
\]
since $C\left(\frac{1}{\alpha_{1}}-\frac{1}{\alpha}\right)=-2+\left(\frac{1}{\alpha_{1}}-\frac{1}{\alpha}\right)$.
Thus for $k>1$ we have 
\begin{equation}
\sigma_{\boldsymbol{\alpha}}(x)=\alpha\frac{x^{\frac{1}{\alpha}}\left(\log\log x\right)^{k-1}}{(k-1)!\log x}\sum_{m\in\mathcal{P}_{\boldsymbol{\beta}}^{\sigma}}m^{-\frac{1}{\alpha}}+O\left(x^{\frac{1}{\alpha}}\frac{\left(\log\log x\right)^{k-2}}{\log x}\right),\label{eq:Main Theorem with error k>1}
\end{equation}
and for $k=1$,
\begin{equation}
\sigma_{\boldsymbol{\alpha}}(x)=\alpha\frac{x^{\frac{1}{\alpha}}}{\log x}\sum_{m\in\mathcal{P}_{\boldsymbol{\beta}}^{\sigma}}m^{-\frac{1}{\alpha}}+O\left(\frac{x^{\frac{1}{\alpha}}}{\left(\log x\right)^{2}}\right).\label{eq:Main Theorem with Error k=00003D1}
\end{equation}
This yields the desired asymptotic
\begin{equation}
\sigma_{\boldsymbol{\alpha}}(x)\sim\alpha\frac{x^{\frac{1}{\alpha}}\left(\log\log x\right)^{k-1}}{(k-1)!\log x}\sum_{m\in\mathcal{P}_{\boldsymbol{\beta}}^{\sigma}}m^{-\frac{1}{\alpha}},\label{eq:asymptotic log for sigma alpha}
\end{equation}
and since 
\[
\sigma_{k}\left(x^{\frac{1}{\alpha}}\right)\sim\alpha\frac{x^{\frac{1}{\alpha}}\left(\log\log x\right)^{k-1}}{(k-1)!\log x}
\]
by Landau's estimates \ref{eq:pi/sigma Landau asymptotic}, we conclude
that 
\begin{equation}
\sigma_{\boldsymbol{\alpha}}(x)\sim\sigma_{k}\left(x^{\frac{1}{\alpha}}\right)\sum_{m\in\mathcal{P}_{\boldsymbol{\beta}}^{\sigma}}m^{-\frac{1}{\alpha}},\label{eq:Main theorem asymp sigma style}
\end{equation}
proving the first part of Theorem \ref{thm: Main Theorem}.

\subsection{$\pi_{\boldsymbol{\alpha}}(x)$}

To prove the same result for $\pi_{\boldsymbol{\alpha}}(x)$, we start
again by splitting integers $n\in\mathcal{P}_{\boldsymbol{\alpha}}^{\sigma}$
into two parts, one in $\mathcal{P}_{k}^{\pi}$, and one in $\mathcal{P}_{\boldsymbol{\beta}}^{\pi}$.
With this in mind we consider 
\[
\sum_{\begin{array}{c}
n^{\alpha}m\leq x\\
n\in\mathcal{P}_{k}^{\pi},m\in\mathcal{P}_{\boldsymbol{\beta}}^{\pi}
\end{array}}1.
\]
This will be strictly larger then $\pi_{\boldsymbol{\alpha}}(x)$
since $n$ and $m$ may have prime factors in common. (Note that since
all factors are distinct, we cannot have multiple representations
$k=mn$.) However, we can throw out all of the terms for which $\gcd\left(m,n\right)>1$
without affecting the asymptotic. Write $n=q_{1}\cdots q_{k}$, and
$m=p_{1}^{\alpha_{1}}\cdots p_{r}^{\alpha_{r}}$. If $\gcd\left(m,n\right)>1$,
then we must have $q_{i}=p_{j}$ for some $i,j$. The set of all tuples
with $q_{i}=p_{j}$ is bounded above by 
\[
\sigma_{\boldsymbol{\alpha}_{i,j}}(x)
\]
where $\boldsymbol{\alpha}_{i,j}=\left(\alpha,\dots,\alpha,\alpha_{1},\cdots,(\alpha_{j}+\alpha),\cdots,\alpha_{r}\right)\in\mathbb{N}^{k-1+r}$
and we have $k-1$ copies of $\alpha$. In particular, by \ref{eq:asymptotic log for sigma alpha},
we see that 
\[
\sigma_{\boldsymbol{\alpha}_{i,j}}(x)=O\left(x^{\frac{1}{\alpha}}\frac{\left(\log\log x\right)^{k-2}}{\log x}\right)
\]
for $k>1$, and 
\[
\sigma_{\boldsymbol{\alpha}_{i,j}}(x)=O_{\epsilon}\left(x^{\frac{1}{\alpha_{1}}+\epsilon}\right)
\]
for any $\epsilon>0$ when $k=1$. Since there are at most $k\cdot r$
possible pairs $\left(i,j\right)$, it follows that for $k>1$ 
\[
\pi_{\boldsymbol{\alpha}}(x)=\sum_{\begin{array}{c}
n^{\alpha}m\leq x\\
n\in\mathcal{P}_{k}^{\pi},m\in\mathcal{P}_{\boldsymbol{\beta}}^{\pi}
\end{array}}1+O\left(x^{\frac{1}{\alpha}}\frac{\left(\log\log x\right)^{k-2}}{\log x}\right),
\]
 and a similar error term as before when $k=1$. The main term may be rewritten as 
\[
\sum_{\begin{array}{c}
m\leq x\\
m\in\mathcal{P}_{\boldsymbol{\beta}}^{\pi}
\end{array}}\sum_{\begin{array}{c}
n^{\alpha}\leq \frac{x}{m}\\
n\in\mathcal{P}_{k}^{\pi}
\end{array}}1=\sum_{\begin{array}{c}
m\leq x\\
m\in\mathcal{P}_{\boldsymbol{\beta}}^{\pi}
\end{array}}\pi_{k}\left(\left(\frac{x}{m}\right)^{\frac{1}{\alpha}}\right),
\]
and from here, following through the exact same sequence of steps
and lemmas from the previous section will yield 
\[
\sum_{\begin{array}{c}
m\leq x\\
m\in\mathcal{P}_{\boldsymbol{\beta}}^{\pi}
\end{array}}\pi_{k}\left(\left(\frac{x}{m}\right)^{\frac{1}{\alpha}}\right)\sim\alpha\frac{x^{\frac{1}{\alpha}}\left(\log\log x\right)^{k-1}}{(k-1)!\log x}\sum_{m\in\mathcal{P}_{\boldsymbol{\beta}}^{\pi}}m^{-\frac{1}{\alpha}}.
\]
All of the upper bounds for $\sigma_{\boldsymbol{\alpha}}(x)$ still
apply to $\pi_{\boldsymbol{\alpha}}(x)$, and the only change is that
we are summing over $\mathcal{P}_{\boldsymbol{\beta}}^{\pi}$ rather
then $\mathcal{P}_{\boldsymbol{\beta}}^{\sigma}$, which is why the
final sum is different. Using \ref{eq:pi/sigma Landau asymptotic},
we get that 
\begin{equation}
\pi_{\boldsymbol{\alpha}}(x)\sim\pi_{k}\left(x^{\frac{1}{\alpha}}\right)\sum_{m\in\mathcal{P}_{\boldsymbol{\beta}}^{\pi}}m^{-\frac{1}{\alpha}},\label{eq:pi main theorem asymptotic}
\end{equation}
proving the second part of Theorem \ref{thm: Main Theorem}. If the
error term is kept throughout the above computations, we get the more
precise 
\begin{equation}
\pi_{\boldsymbol{\alpha}}(x)=\alpha\frac{x^{\frac{1}{\alpha}}\left(\log\log x\right)^{k-1}}{(k-1)!\log x}\sum_{m\in\mathcal{P}_{\boldsymbol{\beta}}^{\pi}}m^{-\frac{1}{\alpha}}+O\left(x^{\frac{1}{\alpha}}\frac{\left(\log\log x\right)^{k-2}}{\log x}\right)\label{eq:main theorem pi with error k>1}
\end{equation}
when $k>1$, and 
\begin{equation}
\pi_{\boldsymbol{\alpha}}(x)=\alpha\frac{x^{\frac{1}{\alpha}}\left(\log\log x\right)^{k-1}}{(k-1)!\log x}\sum_{m\in\mathcal{P}_{\boldsymbol{\beta}}^{\pi}}m^{-\frac{1}{\alpha}}+O\left(\frac{x^{\frac{1}{\alpha}}}{\log^{2}x}\right),\label{eq:main theorem pi with error k=00003D1}
\end{equation}
for $k=1$.

\section{The Constant Factor}

Let $\alpha>0$ be given,  let $A=\left\{ \alpha_{1},\cdots,\alpha_{r}\right\} $
where $\alpha<\alpha_{i}\leq\alpha_{j}$ for all $i,j$, and set
set $\boldsymbol{\beta}=\left(\alpha_{1},\dots,\alpha_{r}\right)$.
If every $n\in\mathcal{P}_{\boldsymbol{\beta}}^{\sigma}$ has one
and only one representation of the form $n=p_{1}^{\alpha_{1}}\cdots p_{r}^{\alpha_{r}}$,
then we may decompose the sum as 
\[
\sum_{n\in\mathcal{P}_{\boldsymbol{\beta}}^{\sigma}}n^{-\frac{1}{\alpha}}=\sum_{p_{1}}\sum_{p_{2}}\cdots\sum_{p_{r}}\left(p_{1}^{\alpha_{1}}\cdots p_{r}^{\alpha_{r}}\right)^{-\frac{1}{\alpha}}.
\]
This equals 
\[
\left(\sum_{p_{1}}p_{1}^{-\frac{\alpha_{1}}{\alpha}}\right)\cdots\left(\sum_{p_{r}}p_{r}^{-\frac{\alpha_{r}}{\alpha}}\right)
\]
which by definition of the prime zeta function, $P(s)=\sum_{p}p^{-s}$,
is 
\[
\prod_{i=1}^{r}P\left(\frac{\alpha_{i}}{\alpha}\right).
\]
We now show that each integer can be uniquely represented if and only
if $\sum_{i}\epsilon_{i}\alpha_{i}=0$ with $\epsilon_{i}\in\left\{ -1,0,1\right\} $
implies that every $\epsilon_{i}=0$. Suppose we are given $\epsilon_{i}$,
not all zero, with $\sum_{i}\epsilon_{i}\alpha_{i}=0$. Then we have
then we have $\alpha_{i_{1}}+\alpha_{i_{2}}+\cdots+\alpha_{i_{k}}=\alpha_{j_{1}}+\alpha_{j_{2}}+\cdots+\alpha_{j_{l}}=M$
for some $M$ where each all of the $i_{n}$ and $j_{m}$ are distinct.
Setting $p_{i_{1}}=\cdots=p_{i_{k}}=p$, and $p_{j_{1}}=\cdots=p_{j_{l}}=q$,
we will have a factor of $q^{M}p^{M}$, and this allows us to permute
$q$ and $p$ giving two representations of the same integer. Conversely,
if we have two representations of the same integer, then it must be
because of a factor of the form $q^{M}p^{M}$, which implies that
we must have $\sum_{i}\epsilon_{i}\alpha_{i}=0$ for some non zero
choices $\epsilon_{i}$. This then completes the proof of Theorem
\ref{thm:Theorem 2}.

\end{document}